\def\ArtDir{main/}%% Trim Size: 9.75in x 6.5in
\documentclass{ws-ijprai}
\def\ArtDir{main/}
\newtheorem{Theorem}{Theorem}

\newtheorem{Definition}{Definition}
\newtheorem{cor}[Theorem]{Corollary}
\newtheorem{Lemma}{Lemma}

\begin{document}

\markboth{HyperKahler Contact Distributions}{H. Attarchi and F. Babaei}

%%%%%%%%%%%%%%%%%%%%% Publisher's area please ignore %%%%%%%%%%%%%%%
%
\catchline{}{}{}{}{}
%
%%%%%%%%%%%%%%%%%%%%%%%%%%%%%%%%%%%%%%%%%%%%%%%%%%%%%%%%%%%%%%%%%%%%
\title{HyperKahler Contact Distributions}
%\title{Instructions for Typesetting Camera-Ready \\
%Manuscripts using \TeX\ or \LaTeX\footnote{For the title, try not to use
%more than 3 lines. Typeset the title in 10 pt Roman, boldface with the first letter of important words capitalized.}}

\author{Hassan Attarchi}%\footnote{Typeset names in 8~pt Roman. Use the footnote to indicate the present or permanent address of the author.}}

\address{Department of Mathematics, University of California, Riverside,\\
Riverside, CA 92521, USA\\
\email{hassan.attarchi@ucr.edu}
\http{profiles.ucr.edu/app/home/profile/hassana}}

\author{Fatemeh Babaei}

\address{Department of Mathematics and Computer Science,\\
Amirkabir University of Technology, IR\\
\email{E-mail\,$:$ f.babaei@aut.ac.ir}}

\maketitle

%begin{history}
%received{(Day Month Year)}
%revised{(Day Month Year)}
%accepted{(Day Month Year)}
%comby{(xxxxxxxxxx)}
%end{history}

\begin{abstract}
Let $(\varphi_\alpha,\xi_\alpha,g)$ for $\alpha=1,2$, and $3$ be a contact metric $3$-structure on the manifold $M^{4n+3}$. We show that the $3$-contact distribution of this structure admits a HyperKahler structure whenever $(M^{4n+3},\varphi_\alpha,\xi_\alpha,g)$ is a $3$-Sasakian manifold. In this case, we call it HyperKahler contact distribution. To analyze the curvature properties of this distribution, we define a special metric connection that is completely determined by the HyperKahler contact distribution. We prove that the $3$-Sasakian manifold is of constant $\varphi_{\alpha}$-sectional curvatures if and only if its HyperKahler contact distribution has constant holomorphic sectional curvatures.
\end{abstract}


%%%%%%%%%%%%%%%%%%%%%%%%%%%%%%%%%%%%%%%%%%%%%%%%%%%%%%%%%%%%%%%%%%%%%%%%%
\section{Introduction}
In 1960, Sasaki \cite{sas1} introduced a geometric structure related to an almost contact structure. This structure has known as Sasakian structure, and it has been studied extensively ever since as an odd-dimensional analogous of Kahler manifolds \cite{yano}.\par
Likewise to the concept of Quaternionic Kahler and HyperKahler manifolds in Quaternion spaces \cite{at,boyer1,boyer2,cap1,hi,ishi}, the contact $3$-structure were introduced on $4n+3$-dimensional manifolds. A Sasakian $3$-structure, that have a close relation to both HyperKahler and Quaternionic Kahler manifolds, first appeared in a paper by Kuo in \cite{kuo} and, independently, by Udriste in \cite{udri}. In 1970, more papers were published in the Japanese literature discussing Sasakian $3$-structure, see \cite{KT70,ty,ta}. Later, in 1973, Ishihara \cite{ishi1} had shown that if the distribution formed by the three Killing vector fields which define the Sasakian $3$-structure is regular then the space of leaves is a quaternionic Kahler manifold.\par
We assume that $(M,\varphi_{\alpha},\xi_{\alpha},g)$ for $\alpha=1,2$, and $3$ is a $3$-Sasakian manifold, and $\mathbf{H}$ denotes the transverse distribution to the Riemannian foliation generated by $\{\xi_1,\xi_2,\xi_3\}$ with respect to the metric $g$ in the tangent bundle $TM$. The purpose of this paper is to study the geometric properties of the distribution $\mathbf{H}$ such as curvature tensor, sectional curvature and Ricci tensor. We investigate the close relation of $\varphi_{\alpha}$-sectional curvatures of $M$ and holomorphic sectional curvatures of $\mathbf{H}$. We can refer to \cite{bejancu2} as an analogue work on Sasakian manifolds.\par
Aiming at our purpose, we organized this paper as follows. In section 2, we present some basic notations and definitions which are needed in the following sections. In section 3, the new linear connection $\bar{\nabla}$ is introduced in the terms of Levi-Civita connection. We show that $\bar{\nabla}$ is a metric connection and completely determined by $\mathbf{H}$. Moreover, we prove that the $\varphi_{\alpha}$ structures are parallel with respect to $\bar{\nabla}$ on $\mathbf{H}$. In section 4, we present the HyperKahler properties of the distribution $\mathbf{H}$ and call it the HyperKahler contact distribution. In this section, the curvature and Ricci tensor of HyperKahler contact distribution $\mathbf{H}$ is defined with respect to the metric connection $\bar{\nabla}$. Finally, in section 4, we prove some theorems that show the curvature properties of this distribution and its close geometric relation with the $3$-Sasakian manifold.
%%%%%%%%%%%%%%%%%%%%%%%%%%%%%%%%%%%%%%%%%%%%%%%%%%%%%%%%%%%%%%%%%%%%%%%%%%%%%%%%%%%%%%%%%%%%%%%%%%%%%%%%%%%%%
\section{Preliminaries and Notations}
Let $M$ be a $(2n+1)$-dimensional smooth manifold. Then, the structure $(\varphi, \eta, \xi)$ on $M$, consisting of $(1,1)$-tensor $\varphi$, non-vanishing vector field $\xi$ and 1-form $\eta$, is called an \emph{almost contact structure} if
$$\varphi^2=-I+\eta\otimes\xi, \ and\ \ \eta(\xi)=1.$$
This structure will be called a \emph{contact structure} if
$$\eta\wedge(d\eta)^n\neq0.$$
The manifold $M^{2n+1}$ with the (almost) contact structure $(\varphi, \eta, \xi)$ is called (almost) contact manifold, and it is denoted by $(M, \varphi, \eta, \xi)$ \cite{Blair}. It was proved that all almost contact manifolds admit a compatible Riemannian metric in the following sense
\begin{equation}\label{compat}
\eta(X)=g(\xi,X), \ \ g(\varphi X,\varphi Y)=g(X,Y)-\eta(X)\eta(Y),
\end{equation}
for all $X,Y\in\Gamma(TM)$. In case of contact metric manifolds, the fundamental 2-form $\Omega$ defined by
$$\Omega(X,Y)=g(X,\varphi Y),\ \ \ \ \forall X,Y\in\Gamma(TM)$$
coincides with $d\eta$.\par
Let $\nabla$ be the Levi-Civita connection with respect to the metric $g$ on the contact manifold $(M,\varphi,\eta,\xi,g)$. Then, the (almost) contact metric manifold $(M,\varphi,\eta,\xi,g)$ is a \emph{Sasakian manifold} if
$$(\nabla_X\varphi)Y=g(X,Y)\xi-\eta(Y)X.$$
If $(M,\varphi,\eta,\xi,g)$ is a Sasakian manifold then,
\begin{equation}\label{sas pro}
\begin{split}
\nabla_X\xi= & -\varphi X,\ \ \ \ \  R(X,Y)\xi=\eta(Y)X-\eta(X)Y,\\
S(X,\xi)= & 2n\eta(X),\ \ \ R(\xi,X)Y=g(X,Y)\xi-\eta(Y)X,
\end{split}
\end{equation}
where $R$ and $S$ are curvature and Ricci tensor, respectively, given by following formulas
\begin{equation}\label{cur1}
R(X,Y)Z:=\nabla_X\nabla_YZ-\nabla_Y\nabla_XZ-\nabla_{[X,Y]}Z,
\end{equation}
\begin{equation}\label{ric1}
S(X,Y):=\sum_{i=1}^{2n+1}R(E_i,X,E_i,Y)=\sum_{i=1}^{2n+1}g(R(E_i,X)Y,E_i),
\end{equation}
where $E_i$ are orthonormal local vector fields on $(M,g)$ \cite{oku,tan}.\par
Let $(M,g)$ be a smooth Riemannian manifold of dimension $4n+3$. The manifold $(M,g)$ is called a $3$-\emph{Sasakian manifold} when it is endowed with three Sasakian structures $(M,\varphi_{\alpha},\eta^{\alpha},\xi_{\alpha},g)$ for $\alpha=1,2,3$, satisfying the following relations
$$\begin{array}{l}
\hspace{.5cm}\varphi_{\theta}=\varphi_{\beta}\varphi_{\gamma}-\eta^{\gamma}\otimes\xi_{\beta}=-\varphi_{\gamma}\varphi_{\beta}
+\eta^{\beta}\otimes\xi_{\gamma},\cr
\xi_{\theta}=\varphi_{\beta}\xi_{\gamma}=-\varphi_{\gamma}\xi_{\beta}, \ \ \ \eta^{\theta}=\eta^{\beta}\circ\varphi_{\gamma}=-\eta^{\gamma}\circ\varphi_{\beta},
\end{array}$$
for all even permutations $(\beta,\gamma,\theta)$ of $(1,2,3)$ \cite{Blair}.\par
Let $\mathbf{\xi}$ be the distribution spanned by the three global (Reeb) vector fields $\{\xi_1,\xi_2,\xi_3\}$. By means of $\nabla_X\xi_{\alpha}=-\varphi_{\alpha}X$ for $\alpha=1,2,3$, one can prove the integrability of $\mathbf{\xi}$ as follows
\begin{equation}\label{bracket xi}
[\xi_{\alpha},\xi_{\beta}]=\nabla_{\xi_{\alpha}}\xi_{\beta}-\nabla_{\xi_{\beta}}\xi_{\alpha}=2\xi_{\gamma},
\end{equation}
for all even permutation $(\alpha,\beta,\gamma)$ of $(1,2,3)$. Therefore, $\mathbf{\xi}$ defines a $3$-dimensional foliation on $M$. Moreover, the equations $\nabla_{\xi_{\alpha}}g=0$ for $\alpha=1,2,3$ show that the foliation $\mathbf{\xi}$ is a Riemannian foliation. The transverse distribution of $\xi$ with respect to the metric $g$ is denoted by $\mathbf{H}$, where $\mathbf{H}=\cap_{\alpha=1}^3ker(\eta^{\alpha})$. The distribution $\mathbf{H}$ is a $4n$-dimensional distribution on $M$. Thus, we obtain the following decomposition of the tangent bundle $TM$:
$$TM=\mathbf{H}\oplus\mathbf{\xi}.$$
In a $3$-Sasakian manifold, the distribution $\mathbf{H}$ is never integrable and in the sequel we call it $3$-\emph{contact distribution}. Note that throughout this paper, Latin and Greek indices are used for the ranges $1,2,...,4n$ and $1,2,3$, respectively.
%%%%%%%%%%%%%%%%%%%%%%%%%%%%%%%%%%%%%%%%%%%%%%%%%%%%%%%%%%%%%%%%%%%%%%%%%%%%%%%%%%%%%%%%%%%%%%%%%%%%%%%%%%%%%
\section{$\mathbf{H}$-Connection of $3$-Contact Distribution}
Considering the foliation $\mathbf{\xi}$ on a $3$-Sasakian manifold $M$, we choose the following local coordinate system,
$$\forall \mathbf{x}\in M \ \ \ \ \mathbf{x}=(z^1,z^2,z^3,x^1,...,x^{4n}),$$
where $\xi_\alpha=\frac{\partial}{\partial z^\alpha}$. Then, one can construct the local basis $\{\frac{\delta}{\delta x^1},...,\frac{\delta}{\delta x^{4n}}\}$ of $\mathbf{H}$ orthogonal to $\mathbf{\xi}$ with respect to the metric $g$ where
$$\frac{\delta}{\delta x^i}=\frac{\partial}{\partial x^i}-\eta_i^{\alpha}\xi_{\alpha}, \ \ i=1,...,4n\ \ \alpha=1,2,3$$
and $\eta_i^{\alpha}=\eta^{\alpha}(\frac{\partial}{\partial x^i})$. Considering the local basis
\begin{equation}\label{basis1}
\{\xi_1,\xi_2,\xi_3,\frac{\delta}{\delta x^1},...,\frac{\delta}{\delta x^{4n}}\},
\end{equation}
the Riemannian metric $g$ will have the following presentation:
\begin{equation}\label{met.mat.}
g:=\left(%
\begin{array}{cccc}
1 & 0 & 0 & 0\\
0 & 1 & 0 & 0\\
0 & 0 & 1 & 0\\
0 & 0 & 0 & [g_{ij}]\\
\end{array}%
\right)
\end{equation}
where $g_{ij}=g(\frac{\delta}{\delta x^i},\frac{\delta}{\delta x^j})$. Moreover, the adapted local frame (\ref{basis1}) will satisfy the following properties:
\begin{equation}\label{bracket1}
\left\{\begin{array}{l}
dx^k([\frac{\delta}{\delta x^i},\xi_{\alpha}])=-2d^2x^k(\frac{\delta}{\delta x^i},\xi_{\alpha})+\frac{\delta}{\delta x^i}(dx^k(\xi_{\alpha}))-\xi_{\alpha}(dx^k(\frac{\delta}{\delta x^i}))=0\cr
\eta^{\beta}([\frac{\delta}{\delta x^i},\xi_{\alpha}])=-2d\eta^{\beta}(\frac{\delta}{\delta x^i},\xi_{\alpha})+\frac{\delta}{\delta x^i}(\eta^{\beta}(\xi_{\alpha}))-\xi_{\alpha}(\eta^{\beta}(\frac{\delta}{\delta x^i}))=0
\end{array}\right.
\end{equation}
\begin{equation}\label{bracket4}
\left\{\begin{array}{l}
dx^k([\frac{\delta}{\delta x^i},\frac{\delta}{\delta x^j}])=\cr
\hspace{1.7cm}-2d^2x^k(\frac{\delta}{\delta x^i},\frac{\delta}{\delta x^j})+\frac{\delta}{\delta x^i}(dx^k(\frac{\delta}{\delta x^j}))-\frac{\delta}{\delta x^j}(dx^k(\frac{\delta}{\delta x^i}))=0\cr
\eta^{\beta}([\frac{\delta}{\delta x^i},\frac{\delta}{\delta x^j}])=\cr
\hspace{1.7cm}-2d\eta^{\beta}(\frac{\delta}{\delta x^i},\frac{\delta}{\delta x^j})+\frac{\delta}{\delta x^i}(\eta^{\beta}(\frac{\delta}{\delta x^j}))-\frac{\delta}{\delta x^j}(\eta^{\beta}(\frac{\delta}{\delta x^i}))=-2\Omega_{ij}^{\beta}
\end{array}\right.
\end{equation}
Using (\ref{bracket1}) and (\ref{bracket4}), we obtain
\begin{equation}\label{bracket2}
[\frac{\delta}{\delta x^i},\xi_{\alpha}]=0,\ \ \ [\frac{\delta}{\delta x^i},\frac{\delta}{\delta x^j}]=-2\Omega_{ij}^{\alpha}\xi_{\alpha}.
\end{equation}
Let,
$$h_{\alpha\beta}(X)=\frac{1}{2}\left(\mathcal{L}_{\xi_{\alpha}}\varphi_{\beta}\right)(X),$$
for $\alpha, \beta=1,2,3$. Then
$$\begin{array}{l}
h_{11}(X)=h_{22}(X)=h_{33}(X)=0,\cr
h_{12}(X)=-h_{21}(X)=\varphi_3(X),\cr
h_{31}(X)=-h_{13}(X)=\varphi_2(X),\cr
h_{23}(X)=-h_{32}(X)=\varphi_1(X),
\end{array}$$
for all $X\in\Gamma TM$.
\begin{Theorem}\label{levi1}
	In the adapted basis (\ref{basis1}), the Levi-Civita connection $\nabla$ with respect to the Riemannian metric $g$ of the $3$-Sasakian manifold $(M,\varphi_{\alpha},\eta^{\alpha},\xi_{\alpha},g)$ has the following components:
	$$\left\{\begin{array}{l}
	\nabla_{\frac{\delta}{\delta x^i}}\frac{\delta}{\delta x^j}=F_{ij}^k\frac{\delta}{\delta x^k}-\Omega_{ij}^{\alpha}\xi_{\alpha},\cr
	\nabla_{\xi_{\alpha}}\frac{\delta}{\delta x^i}=\nabla_{\frac{\delta}{\delta x^i}}\xi_{\alpha}=\Omega_{ij}^{\alpha}g^{jk}\frac{\delta}{\delta x^k},\cr
	\nabla_{\xi_1}\xi_2=-\nabla_{\xi_2}\xi_1=\xi_3,\cr
	\nabla_{\xi_3}\xi_1=-\nabla_{\xi_1}\xi_3=\xi_2,\cr
	\nabla_{\xi_2}\xi_3=-\nabla_{\xi_3}\xi_2=\xi_1,\cr
	\nabla_{\xi_1}\xi_1=-\nabla_{\xi_2}\xi_2=\nabla_{\xi_3}\xi_3=0,
	\end{array}\right.$$
	where
	\begin{equation}\label{gamma}
	F_{ij}^k=\frac{g^{kh}}{2}\left\{\frac{\delta g_{ih}}{\delta x^j}+\frac{\delta g_{jh}}{\delta x^i}-\frac{\delta g_{ij}}{\delta x^h}\right\}.
	\end{equation}
\end{Theorem}
\begin{cor}\label{bundle like}
	\cite{boyer1,boyer2} The foliation $\mathbf{\xi}$ is totally geodesic, and the Riemannian metric $g$ is bundle-like with respect to this foliation.
\end{cor}
Consider the Levi-Civita connection $\nabla$ on the $3$-Sasakian manifold $(M,g)$. Then, we define the linear connection $\bar{\nabla}$ by
\begin{equation}\label{new conn}
\bar{\nabla}_XY=\nabla_XY-\eta^{\alpha}(X)\nabla_Y\xi_{\alpha}-\eta^{\alpha}(Y)\nabla_X\xi_{\alpha}+
\Omega^{\alpha}(X,Y)\xi_{\alpha}.
\end{equation}
Using (\ref{basis1}) and (\ref{new conn}), we obtain the following theorem:
\begin{Theorem}\label{new comp.}
	The linear connection $\bar{\nabla}$ is completely determined by
	$$\left\{\begin{array}{l}
	\bar{\nabla}_{\frac{\delta}{\delta x^i}}\frac{\delta}{\delta x^j}=F_{ij}^k\frac{\delta}{\delta x^k},\cr
	\bar{\nabla}_{\xi_{\alpha}}\frac{\delta}{\delta x^i}=\bar{\nabla}_{\frac{\delta}{\delta x^i}}\xi_{\alpha}=0,\cr
	\bar{\nabla}_{\xi_{\alpha}}\xi_{\beta}=0.
	\end{array}\right.$$
\end{Theorem}
From Theorem \ref{new comp.} and Eq. (\ref{gamma}), we obtain that $\bar{\nabla}$ is completely determined by Riemannian metric induced by the $g$ on the $3$-contact distribution $\mathbf{H}$. Moreover, it is easy to see that the $3$-contact distribution $\mathbf{H}$ and foliation $\mathbf{\xi}$ are parallel with respect to the linear connection $\bar{\nabla}$. For these reasons we call $\bar{\nabla}$ the $\mathbf{H}$-\emph{connection} on the $3$-Sasakian manifold $M$.\par
This is surprising and interesting that the connection $\bar{\nabla}$ presented in (\ref{new conn}) coincides with the connection $\tilde{\nabla}$ defined in \cite{cap}. In the following Lemma, we prove this equality
\begin{Lemma}
	Let $(M,g)$ be a Riemannian manifold with a Sasakian $3$-structure. Then, the connection $\bar{\nabla}$ presented in (\ref{new conn}) coincides with the connection $\tilde{\nabla}$ defined in \cite{cap}.
\end{Lemma}
\begin{proof}
	By the definition of $\tilde{\nabla}$ in \cite{cap} and using the local frame (\ref{basis1}), we obtain
	\begin{equation*}
    \begin{split}
	\tilde{\nabla}_{\frac{\delta}{\delta x^i}}\frac{\delta}{\delta x^j}= & (\nabla_{\frac{\delta}{\delta x^i}}\frac{\delta}{\delta x^j})^h=F_{ij}^k\frac{\delta}{\delta x^k}=\bar{\nabla}_{\frac{\delta}{\delta x^i}}\frac{\delta}{\delta x^j},\\
	\tilde{\nabla}_{\xi_{\alpha}}\frac{\delta}{\delta x^i}= & [\xi_{\alpha},\frac{\delta}{\delta x^i}]=0=\bar{\nabla}_{\xi_{\alpha}}\frac{\delta}{\delta x^i},\\
    \tilde{\nabla}_{\frac{\delta}{\delta x^i}}\xi_{\alpha}=& \tilde{\nabla}_{\xi_{\beta}}\xi_{\alpha}=0=\bar{\nabla}_{\frac{\delta}{\delta x^i}}\xi_{\alpha}=\bar{\nabla}_{\xi_{\beta}}\xi_{\alpha},
	\end{split}
	\end{equation*}
	where $(\nabla_{\frac{\delta}{\delta x^i}}\frac{\delta}{\delta x^j})^h$ denotes the restricted component of the Levi-Civita connection $\nabla$ on $3$-contact distribution $\mathbf{H}$.
\end{proof}
From which, one can find the metrizability of $\bar{\nabla}$ and some more information about its torsion and curvature in \cite{cap}. In the following, we present the local expression of the torsion and curvature of $\bar{\nabla}$ with respect to the frame (\ref{basis1}),
\begin{equation}\label{torsion}
\left\{\begin{array}{l}
T_{\bar{\nabla}}(\frac{\delta}{\delta x^i},\frac{\delta}{\delta x^j})=2\Omega_{ij}^{\alpha}\xi_{\alpha},\cr
T_{\bar{\nabla}}(\frac{\delta}{\delta x^i},\xi_{\alpha})=0,\cr
T_{\bar{\nabla}}(\xi_{\alpha},\xi_{\beta})=-T_{\bar{\nabla}}(\xi_{\beta},\xi_{\alpha})=-2\xi_{\gamma},
\end{array}\right.
\end{equation}
for all even permutations $(\alpha,\beta,\gamma)$ of $(1,2,3)$.
\begin{equation}\label{curvature}
\left\{\begin{array}{l}
\bar{R}(\frac{\delta}{\delta x^i},\frac{\delta}{\delta x^j})\frac{\delta}{\delta x^k}=\bar{R}_{\ ijk}^h\frac{\delta}{\delta x^h},\cr
\bar{R}(\frac{\delta}{\delta x^i},\xi_{\alpha})\frac{\delta}{\delta x^j}=-\xi_{\alpha}(F_{ij}^k)\frac{\delta}{\delta x^k},\cr
\bar{R}(\xi_{\alpha},\xi_{\beta})\frac{\delta}{\delta x^i}=0,\cr
\bar{R}(X,Y)\xi_{\alpha}=0,
\end{array}\right.
\end{equation}
where $X,Y\in\Gamma(TM)$ and
\begin{equation}\label{cur. com.}
\bar{R}_{\ ijk}^h=\frac{\delta F_{ij}^h}{\delta x^k}-\frac{\delta F_{ik}^h}{\delta x^j}+F_{ij}^tF_{tk}^h-F_{ik}^tF_{tj}^h.
\end{equation}
Moreover, the Lie brackets of vector fields on $M$ in terms of the $\mathbf{H}$-connection have the following expression
\begin{equation}\label{bracket3}
[X,Y]=\bar{\nabla}_XY-\bar{\nabla}_YX-2\Omega^{\alpha}(X,Y)\xi_{\alpha},
\end{equation}
for any $X,Y\in\Gamma TM$.
\begin{Theorem}\label{varphi}
	Consider the linear connection $\bar{\nabla}$ given by (\ref{new conn}) on $3$-Sasakian manifold $(M,\varphi_{\alpha},\eta^{\alpha},\xi_{\alpha},g)$ for $\alpha=1,2,3$. Then, the following equation is satisfied
	$$(\bar{\nabla}_X\varphi_{\alpha})Y=0,\ \ \ \ \ \ \ \forall\alpha=1,2,3$$
	where $X,Y\in\Gamma\mathbf{H}$.
\end{Theorem}
\begin{proof} To complete the proof, we need to evaluate $(\bar{\nabla}_{\frac{\delta}{\delta x^i}}\varphi_{\alpha})\frac{\delta}{\delta x^j}$. Using (\ref{new conn}), we obtain
	\begin{equation*}
    \begin{split}
	(\bar{\nabla}_{\frac{\delta}{\delta x^i}}\varphi_{\alpha})\frac{\delta}{\delta x^j}= & \bar{\nabla}_{\frac{\delta}{\delta x^i}}\varphi_{\alpha}(\frac{\delta}{\delta x^j})-\varphi_{\alpha}(\bar{\nabla}_{\frac{\delta}{\delta x^i}}\frac{\delta}{\delta x^j})\\
	= & \nabla_{\frac{\delta}{\delta x^i}}\varphi_{\alpha}(\frac{\delta}{\delta x^j})+\Omega^{\beta}(\frac{\delta}{\delta x^i},\varphi_{\alpha}(\frac{\delta}{\delta x^j}))\xi_{\beta}\\
	& -\varphi_{\alpha}(\nabla_{\frac{\delta}{\delta x^i}}\frac{\delta}{\delta x^j}+\Omega^{\beta}(\frac{\delta}{\delta x^i},\frac{\delta}{\delta x^j})\xi_{\beta})\\
	= & (\nabla_{\frac{\delta}{\delta x^i}}\varphi_{\alpha})\frac{\delta}{\delta x^j}+\Omega^{\beta}(\frac{\delta}{\delta x^i},\varphi_{\alpha}(\frac{\delta}{\delta x^j}))\xi_{\beta}-\Omega^{\beta}(\frac{\delta}{\delta x^i},\frac{\delta}{\delta x^j})\varphi_{\alpha}(\xi_{\beta})\\
	= & g(\frac{\delta}{\delta x^i},\frac{\delta}{\delta x^j})\xi_{\alpha}+g(\frac{\delta}{\delta x^i},\varphi_1\varphi_{\alpha}(\frac{\delta}{\delta x^j}))\xi_1-g(\frac{\delta}{\delta x^i},\varphi_1(\frac{\delta}{\delta x^j}))\varphi_{\alpha}(\xi_1)\\
	& +g(\frac{\delta}{\delta x^i},\varphi_2\varphi_{\alpha}(\frac{\delta}{\delta x^j}))\xi_2-g(\frac{\delta}{\delta x^i},\varphi_2(\frac{\delta}{\delta x^j}))\varphi_{\alpha}(\xi_2)\\
	& +g(\frac{\delta}{\delta x^i},\varphi_3\varphi_{\alpha}(\frac{\delta}{\delta x^j}))\xi_3-g(\frac{\delta}{\delta x^i},\varphi_3(\frac{\delta}{\delta x^j}))\varphi_{\alpha}(\xi_3)
\end{split}
\end{equation*}
It is easy to check the last equation is vanish for all $\alpha=1,2,3$. \end{proof}
%%%%%%%%%%%%%%%%%%%%%%%%%%%%%%%%%%%%%%%%%%%%%%%%%%%%%%%%%%%%%%%%%%%%%%%%%%%%%%%%%%%%%%%%%%%%%%%%%%%%%%%%%%%%%
\section{HyperKahler Contact Distribution and its Holomorphic Sectional Curvatures}
If we restrict metric $g$ and $\varphi_{\alpha}$ for $\alpha=1,2,3$ to the $3$-contact distribution $\mathbf{H}$, then $\mathbf{H}$ can be considered as an almost Hyper-Hermitian vector bundle. Moreover, $\varphi_{\alpha}$ for $\alpha=1,2,3$ are parallel with respect to the metric connection $\bar{\nabla}$ on $\mathbf{H}$. Therefore, $\mathbf{H}$ carries an analogue HyperKahler structure and we call it a \emph{HyperKahler contact distribution}. The close relation of HyperKahler and $3$-Sasakian manifolds suggests this name as well.\par
By Corollary \ref{bundle like}, we know that $g_{ij}$ are functions of $(x^i)$ for $i=1,...,4n$ (i.e. $\xi_{\alpha}(g_{ij})=0$ for all $\alpha=1,2,3$). Using this fact and $[\frac{\delta}{\delta x^i},\xi_{\alpha}]=0$, one can check that $\xi_{\alpha}(F_{ij}^k)=0$. Therefore, the equations (\ref{curvature}) and (\ref{cur. com.}) imply the followings
\begin{equation}\label{curvature1}
\left\{\begin{array}{l}
\bar{R}(\frac{\delta}{\delta x^i},\frac{\delta}{\delta x^j})\frac{\delta}{\delta x^k}=\bar{R}_{\ ijk}^h\frac{\delta}{\delta x^h},\cr
\bar{R}(\frac{\delta}{\delta x^i},\xi_{\alpha})\frac{\delta}{\delta x^j}=\bar{R}(\xi_{\alpha},\xi_{\beta})\frac{\delta}{\delta x^i}=\bar{R}(X,Y)\xi_{\alpha}=0,
\end{array}\right.
\end{equation}
where $X,Y\in\Gamma(TM)$ and
\begin{equation}\label{cur. com.1}
\bar{R}_{\ ijk}^h=\frac{\partial F_{ij}^h}{\partial x^k}-\frac{\partial F_{ik}^h}{\partial x^j}+F_{ij}^tF_{tk}^h-F_{ik}^tF_{tj}^h.
\end{equation}
The equations (\ref{curvature1}) and (\ref{cur. com.1}) show that $\bar{R}$ only depends on $\mathbf{H}$. Therefore, we define
\begin{equation}\label{cur H}
\bar{R}(X,Y,Z,W):=g(\bar{R}(X,Y)W,Z), \ \ \ \forall X,Y,Z,W\in\Gamma\mathbf{H},
\end{equation}
and we call it the \emph{curvature tensor field} of $(\mathbf{H},g|_{\mathbf{H}})$.
\begin{Lemma}\label{cur pro}
	The curvature tensor field $\bar{R}$ of the HyperKahler contact distribution $\mathbf{H}$ satisfies the identities:
	\begin{equation}\label{curvature2}
	\left\{\begin{array}{l}
	\bar{R}(X,Y,Z,U)=-\bar{R}(Y,X,Z,U)=-\bar{R}(X,Y,U,Z),\cr
	\bar{R}(X,Y,U,Z)+\bar{R}(Y,Z,U,X)+\bar{R}(Z,X,U,Y)=0,\cr
	\bar{R}(X,Y,Z,U)=\bar{R}(Z,U,X,Y).
	\end{array}\right.
	\end{equation}
	for any $X,Y,Z,U\in\Gamma\mathbf{H}$.
\end{Lemma}
\begin{proof} The first equality is a general property of curvature tensor of any linear connection. The next equality of the first equation is a consequence of the fact that $\bar{\nabla}$ is a metric connection. To prove the second equality, we use (\ref{bracket3}) in a straightforward calculation as follows
	$$\bar{R}(X,Y)Z+\bar{R}(Y,Z)X+\bar{R}(Z,X)Y=\bar{\nabla}_X\bar{\nabla}_YZ-\bar{\nabla}_Y\bar{\nabla}_XZ
	-\bar{\nabla}_{[X,Y]}Z$$
	$$+\bar{\nabla}_Y\bar{\nabla}_ZX-\bar{\nabla}_Z\bar{\nabla}_YX
	-\bar{\nabla}_{[Y,Z]}X+\bar{\nabla}_Z\bar{\nabla}_XY-\bar{\nabla}_X\bar{\nabla}_ZY
	-\bar{\nabla}_{[Z,X]}Y$$
	$$=[X,[Y,Z]]+2(\Omega^{\alpha}(X,[Y,Z])+\Omega^{\alpha}(\bar{\nabla}_XY,Z)+\Omega^{\alpha}(Y,\bar{\nabla}_XZ))
	\xi_{\alpha}$$
	$$+[Y,[Z,X]]+2(\Omega^{\alpha}(Y,[Z,X])+\Omega^{\alpha}(\bar{\nabla}_YZ,X)+\Omega^{\alpha}(Z,\bar{\nabla}_YX))
	\xi_{\alpha}$$
	$$+[Z,[X,Y]]+2(\Omega^{\alpha}(Z,[X,Y])+\Omega^{\alpha}(\bar{\nabla}_ZX,Y)+\Omega^{\alpha}(X,\bar{\nabla}_ZY))
	\xi_{\alpha}=0.$$
	Then, by Lemma 3.1 in \cite{yano} on page 32 the last equality is obtained.
\end{proof}
To present $\bar{R}(X,Y)Z$ in term of $R(X,Y)Z$, for all $X,Y,Z\in\Gamma(TM)$, we compute the followings
\begin{equation}\label{cur com1}
\begin{split}
\bar{\nabla}_X\bar{\nabla}_YZ= & \nabla_X\nabla_YZ\\
& +\eta^{\alpha}(X)\varphi_{\alpha}(\nabla_YZ)+\eta^{\alpha}(Y)\varphi_{\alpha}(\nabla_XZ)+\eta^{\alpha}(Z)\varphi_{\alpha}(\nabla_XY)\\
& +\eta^{\alpha}(\nabla_YZ)\varphi_{\alpha}(X)+\eta^{\alpha}(\nabla_XZ)\varphi_{\alpha}(Y)+\eta^{\alpha}(\nabla_XY)\varphi_{\alpha}(Z)\\
& +\Omega^{\alpha}(X,\nabla_YZ)\xi_{\alpha}+\Omega^{\alpha}(Y,\nabla_XZ)\xi_{\alpha}+\Omega^{\alpha}(\nabla_XY,Z)\xi_{\alpha}\\
& -\left(\Omega^{\alpha}(Y,X)\varphi_{\alpha}(Z)
+\Omega^{\alpha}(Z,X)\varphi_{\alpha}(Y)\right)\\
& +\eta^{\alpha}(Y)\eta^{\beta}(X)\varphi_{\beta}\varphi_{\alpha}(Z)+\eta^{\alpha}(Z)\eta^{\beta}(X)\varphi_{\beta}\varphi_{\alpha}(Y)\\
& -2\sum_{\alpha=1}^3\eta^{\alpha}(Y)\eta^{\alpha}(Z)X+2\eta^{\alpha}(Y)g(X,Z)\xi_{\alpha}\\
& +\eta^{\alpha}(Y)\Omega^{\beta}(X,\varphi_{\alpha}(Z))\xi_{\beta}+\eta^{\alpha}(Z)\Omega^{\beta}(X,\varphi_{\alpha}(Y))\xi_{\beta}\\
& +\eta^{\alpha}(Y)\eta^{\beta}(\varphi_{\alpha}(Z))\varphi_{\beta}(X)+\eta^{\alpha}(Z)\eta^{\beta}(\varphi_{\alpha}(Y))\varphi_{\beta}(X),
\end{split}
\end{equation}
and
\begin{equation}\label{cur com2}
\begin{split}
\bar{\nabla}_{[X,Y]}Z= & \nabla_{[X,Y]}Z+\eta^{\alpha}([X,Y])\varphi_{\alpha}(Z)\\
& +\eta^{\alpha}(Z)\varphi_{\alpha}([X,Y])+\Omega^{\alpha}([X,Y],Z)
\xi_{\alpha}.
\end{split}
\end{equation}
Therefore,
\begin{equation}\label{cur1 2}
\begin{split}
\bar{R}(X,Y)Z= & R(X,Y)Z\\
& -2\Omega^{\alpha}(Y,X)\varphi_{\alpha}(Z)-\Omega^{\alpha}(Z,X)\varphi_{\alpha}(Y)
+\Omega^{\alpha}(Z,Y)\varphi_{\alpha}(X)\\
& +\sum_{\alpha=1}^3(\eta^{\alpha}(X)\eta^{\alpha}(Z)Y-\eta^{\alpha}(Y)\eta^{\alpha}(Z)X)\\
& -\eta^{\alpha}(Z)\eta^{\beta}(Y)\varphi_{\beta}\varphi_{\alpha}(X)+\eta^{\alpha}(Z)\eta^{\beta}(X)
\varphi_{\beta}\varphi_{\alpha}(Y)\\
& +2\eta^{\alpha}(Y)\eta^{\beta}(X)\varphi_{\beta}\varphi_{\alpha}(Z)+2\eta^{\alpha}(Z)\Omega^{\beta}
(X,\varphi_{\alpha}(Y))\xi_{\beta}\\
& -\eta^{\alpha}(X)\Omega^{\beta}(Y,\varphi_{\alpha}(Z))\xi_{\beta}+\eta^{\alpha}(Y)\Omega^{\beta}
(X,\varphi_{\alpha}(Z))\xi_{\beta}\\
& +\eta^{\alpha}(Y)\eta^{\beta}(\varphi_{\alpha}(Z))\varphi_{\beta}(X)+\eta^{\alpha}(Z)\eta^{\beta}
(\varphi_{\alpha}(Y))\varphi_{\beta}(X)\\
& -\eta^{\alpha}(X)\eta^{\beta}(\varphi_{\alpha}(Z))\varphi_{\beta}(Y)-\eta^{\alpha}(Z)\eta^{\beta}(\varphi_{\alpha}(X))\varphi_{\beta}(Y)\\
& +2\eta^{\alpha}(Y)g(X,Z)\xi_{\alpha}-2\eta^{\alpha}(X)g(Y,Z)\xi_{\alpha},
\end{split}
\end{equation}
where the Einstein notation is used for repeated indices $\alpha$ and $\beta$ on their range in case of $\alpha\neq\beta$.
\begin{cor}\label{X X1 X2 X3}
	Let $M$ be a $3$-Sasakian manifold and $(\mathbf{H},g|_{\mathbf{H}})$ the HyperKahler contact distribution on $M$. Then the following holds for each vector field $X\in\Gamma\mathbf{H}$
	$$\bar{R}(X,\varphi_1X,\varphi_2X,\varphi_3X)=R(X,\varphi_1X,\varphi_2X,\varphi_3X).$$
\end{cor}
The \emph{Ricci tensor} $\bar{S}$ of the HyperKahler contact distribution $\mathbf{H}$ is defined by
\begin{equation}\label{Ricci}
\bar{S}(X,Y)=\sum_{i=1}^{4n}\bar{R}(E_i,X,E_i,Y)+\sum_{\alpha=1}^3\bar{R}(\xi_{\alpha},X,\xi_{\alpha},Y)\ \ \ \forall X,Y\in\Gamma\mathbf{H},
\end{equation}
where $\{E_1,...,E_{4n}\}$ is an orthonormal local basis of vectors in $\Gamma\mathbf{H}$.
\begin{Lemma}\label{ric}
	Let $M$ be a connected $3$-Sasakian manifold of dimension $4n+3$. Then, the Ricci tensor $\bar{S}$ of the HyperKahler contact distribution $\mathbf{H}$ satisfies
	$$\bar{S}(X,Y)=(4n+5)g(X,Y) \ \ \ \forall X,Y\in\Gamma\mathbf{H}.$$
\end{Lemma}
\begin{proof}
    In \cite{kashi}, it was proved that $3$-Sasakian manifolds are Einstein spaces and their Ricci tensor fields with respect to the Levi-Civita connection are given by
	$$S(X,Y)=(4n+2)g(X,Y)\ \ \ \forall X,Y\in\Gamma TM,$$
	where $dim(M)=4n+3$. If $\{E_1,...,E_{4n}\}$ is an orthonormal local basis of $\mathbf{H}$ then $\{E_1,...,E_{4n},E_{4n+1}:=\xi_1,E_{4n+2}:=\xi_2,E_{4n+3}:=\xi_3\}$ will be an unitary orthogonal local basis of $TM$, and vice versa.
	Using (\ref{sas pro}), (\ref{cur1 2}) and (\ref{Ricci}), for all $X,Y\in\Gamma\mathbf{H}$, we obtain
	\begin{equation*}
    \begin{split}
	\bar{S}(X,Y)= & \sum_{i=1}^{4n+3}g(\bar{R}(E_i,X)Y,E_i)\\
	= & \sum_{i=1}^{4n+3}g(R(E_i,X)Y,E_i)-3\sum_{\alpha=1}^3\sum_{i=1}^{4n}g(X,\varphi_{\alpha}E_i)g(\varphi_{\alpha}Y,E_i)\\
	& -2\sum_{\beta=1}^3\eta^{\alpha}(\xi_{\beta})g(X,Y)g(\xi_{\alpha},\xi_{\beta})\\
	= & S(X,Y)+9g(X,Y)-6g(X,Y)=(4n+5)g(X,Y).
	\end{split}
	\end{equation*}
\end{proof}
The sectional curvature $K$ of Levi-Civita connection $\nabla$ for the plane $\Pi$ spanned by $\{X,Y\}$ at a point is given by
$$K(\Pi)=K(X,Y)=-\frac{R(X,Y,X,Y)}{g(X,X)g(Y,Y)-g^2(X,Y)}.$$
It is easy to check that $K(\Pi)$ is independent of choosing the vector fields $X$ and $Y$ spanned $\Pi$.\par
The \emph{holomorphic sectional} curvatures of the HyperKahler contact distribution $\mathbf{H}$ are defined by
$$\bar{H}_{\alpha}(X)=\bar{R}(X,\varphi_{\alpha}X,X,\varphi_{\alpha}X) \ \ \ \forall \alpha=1,2,3$$
where $X\in\Gamma\mathbf{H}$ has unit length at any point with respect to the metric $g$.
\begin{Definition}
	Let $M$ be a $3$-Sasakian manifold. The plane $\Pi$ is called a $\varphi_{\alpha}$-plane, whenever for any $X\in\Pi$, the sections $X$ and $\varphi_{\alpha}X$ span $\Pi$.
\end{Definition}
\begin{Theorem}\label{sec rela}
	Let $M$ be a $3$-Sasakian manifold with HyperKahler contact distribution $\mathbf{H}$ and $\alpha$ be a number in $\{1,2,3\}$. Then the $\varphi_{\alpha}$-holomorphic sectional curvatures of $\mathbf{H}$ is equal to $k$ (i.e. $\bar{H}_{\alpha}(X)=k$) if and only if its respective $\varphi_{\alpha}$-sectional curvature of $M$ for the $\varphi_{\alpha}$-plane $\{X,\varphi_{\alpha}X\}$, where $X\in\Gamma\mathbf{\mathbf{H}}$, is equal to $k-3$.
\end{Theorem}
\begin{proof} By using (\ref{cur1 2}) and a straightforward calculation, the result is obtained.\end{proof}
\begin{cor}\label{bianchi}
	Let $M$ be a $3$-Sasakian manifold. Then the holomorphic sectional curvatures $\bar{H}_{\alpha}$ satisfy the following equality
	\begin{equation}\label{sum H}
	\bar{H}_1(X)+\bar{H}_2(X)+\bar{H}_3(X)=12,
	\end{equation}
	for all unitary vector field $X\in\mathbf{H}$.
\end{cor}
\begin{proof} In \cite{tan3}, it was proved that
	\begin{equation}\label{tanno}
	\sum_{\alpha=1}^3H_{\alpha}(X)=3,
	\end{equation}
	where $X$ is a unitary vector field tangent to $\mathbf{H}$ and $H_{\alpha}$ are holomorphic sectional curvatures of Levi-Civita connection on $M$ with respect to $\varphi_{\alpha}$ for $\alpha=1,2,3$. Therefore, equation (\ref{sum H}) is a direct consequence of Theorem \ref{sec rela} and (\ref{tanno}).\end{proof}
\begin{cor}\label{cons2}
	Let $M$ be a $3$-Sasakian manifold. If two holomorphic sectional curvatures of $\mathbf{H}$ are constant then the third one will be constant.
\end{cor}
\begin{Theorem}\label{sec}
	Let $\Pi$ be the $\varphi_{\alpha}$-plane at a point of a $3$-Sasakian manifold $M$. Then, the sectional curvatures of $\Pi$ with respect to the $\nabla$ and $\bar{\nabla}$ are related as follows
	\begin{equation*}
    \begin{split}
	\bar{K}(\Pi)= & K(\Pi)+3+4\left(\eta^{\beta}(X)\eta^{\gamma}(X)\right)^2\\
	& +6\left((\eta^{\beta}(X))^4+(\eta^{\gamma}(X))^4\right)
	-8\left((\eta^{\beta}(X))^2+(\eta^{\gamma}(X))^2\right),
	\end{split}
	\end{equation*}
	where $X$ is a unit vector in the $\varphi_{\alpha}$-plane $\Pi$ and $(\alpha, \beta, \gamma)$ is an arbitrary permutation of $(1,2,3)$.
\end{Theorem}
\begin{proof}  Without losing the generality, we prove these theorem for $\alpha=1$. Consider the $\alpha_1$-plane $\Pi$ and unit vector $X\in\Pi$, then by using (\ref{cur1 2}), we obtain
\begin{equation*}
    \begin{split}
    \bar{K}(\Pi)= & \bar{R}(X,\varphi_1X,X,\varphi_1X)\\
    = & R(X,\varphi_1X,X,\varphi_1X)-3\Omega^{\beta}(\varphi_1X,X)g(\varphi_{\beta}\varphi_1X,X)\\
    & +\Omega^{\beta}(\varphi_1X,\varphi_1X)
	g(\varphi_{\beta}X,X)\\
    & +\sum_{\beta=1}^3\left(\eta^{\beta}(X)\eta^{\beta}(\varphi_1X)g(\varphi_1X,X)-\eta^{\beta}(\varphi_1X)
	\eta^{\beta}(\varphi_1X)g(X,X)\right)\\
    & -\eta^{\beta}(\varphi_1X)\eta^{\gamma}(\varphi_1X)g(\varphi_{\gamma}\varphi_{\beta}X,X)+3\eta^{\beta}(\varphi_1X)
	\eta^{\gamma}(X)g(\varphi_{\gamma}\varphi_{\beta}\varphi_1X,X)\\
	& +2\eta^{\beta}(\varphi_1X)\eta^{\gamma}(\varphi_{\beta}\varphi_1X)g(\varphi_{\gamma}X,X)-\eta^{\beta}(X)
	\eta^{\gamma}(\varphi_{\beta}\varphi_1X)g(\varphi_{\gamma}\varphi_1X,X)\\
	& -\eta^{\beta}(\varphi_1X)\eta^{\gamma}(\varphi_{\beta}X)g(\varphi_{\gamma}\varphi_1X,X)+\eta^{\beta}(\varphi_1X)
	g(X,\varphi_1X)g(\xi_{\beta},X)\\
	& -\eta^{\beta}(X)g(\varphi_1X,\varphi_1X)g(\xi_{\beta},X)+3\eta^{\beta}(\varphi_1X)\Omega^{\gamma}
	(X,\varphi_{\beta}\varphi_1X)g(\xi_{\gamma},X)\\
	& -\eta^{\beta}(X)\Omega^{\gamma}(\varphi_1X,\varphi_{\beta}\varphi_1X)g(\xi_{\gamma},X)\\
	= & K(\Pi)+3+4\left(\eta^{\beta}(X)\eta^{\gamma}(X)\right)^2\\
	& +6\left((\eta^{\beta}(X))^4+(\eta^{\gamma}(X))^4\right)-8\left((\eta^{\beta}(X))^2+(\eta^{\gamma}(X))^2\right).
    \end{split}
\end{equation*}
\end{proof}
%\acknowledgements{\rm We thank the referees for their time and comments.}
%%%%%%%%%%%%%%%%%%%%%%%%%%%%%%%%%%%%%%%%%%%%%%%%%%%%%%%%%%%%%%%%%%%%%%%%%%%%%%%%%%%%%%%%%%%%%%%%%%%%%%%%%%%%%

\end{document}